\documentclass[11pt,reqno]{amsart}

\usepackage{amscd,amssymb,amsmath,amsthm}
\usepackage[arrow,matrix]{xy}
\usepackage{graphicx}
\usepackage{epstopdf}
\usepackage{color}
\newtheorem{thm}{Theorem}

\newtheorem{lemma}{Lemma}

\newtheorem{rk}{Remark}
\newtheorem{cor}{Corollary}

\numberwithin{equation}{section} 

\def\r{\rho}

\def\C{\mathbb C}

\def\C{\mathbb{C}}
\def\N{\mathbb N}

\begin{document}

\title[$p$-adic dynamical systems of a $(2,2)$-rational function]{$p$-adic dynamical systems of $(2,2)$-rational functions with unique fixed
point}

\author{U.A. Rozikov, I.A. Sattarov}

 \address{U.\ A.\ Rozikov \\ Institute of mathematics,
29, Do'rmon Yo'li str., 100125, Tashkent, Uzbekistan.} \email
{rozikovu@yandex.ru}

 \address{I.\ A.\ Sattarov \\ Institute of mathematics,
29, Do'rmon Yo'li str., 100125, Tashkent, Uzbekistan.} \email
{iskandar1207@rambler.ru}

\begin{abstract}

We consider a family of $(2,2)$-rational functions given on the set of complex
$p$-adic field $\C_p$. Each such function has a unique fixed point. We study
$p$-adic dynamical systems generated by the $(2,2)$-rational functions. 
We show that the fixed point is indifferent and therefore
the convergence of the trajectories is not the typical
case for the dynamical systems.
Siegel disks of these dynamical systems are found.
 We obtain an upper bound for the set of limit points of each trajectory, i.e.,
we determine a sufficiently small set containing the set of limit points.
 For each $(2,2)$-rational function on $\C_p$ there are two points
 $\hat x_1=\hat x_1(f)$, $\hat x_2=\hat x_2(f)\in \C_p$ which are
 zeros of its denominator. We give explicit formulas of radiuses of spheres
(with the center at the fixed point) containing some points such that
the trajectories (under actions of $f$) of the points after a finite step
come to $\hat x_1$ or $\hat x_2$. Moreover for a class of $(2,2)$-rational functions
we study ergodicity properties of the dynamical systems on the set of $p$-adic numbers $Q_p$.
For each such function we describe all possible invariant spheres. We show that 
the $p$-adic dynamical system reduced on each invariant sphere is not ergodic with respect to Haar measure. 
\end{abstract}

 \maketitle

\section{Introduction}

We study dynamical systems generated by a rational
function. A function is called a $(n,m)$-rational function if and only if it
can be written in the form $f(x)={P_n(x)\over Q_m(x)}$, where
$P_n(x)$ and $Q_m(x)$ are polynomial functions with degree $n$ and
$m$ respectively, $Q_m(x)$ is not the zero polynomial.

It is known that analytic functions play a fundamental
role in complex analysis and rational functions play
an analogous role in $p$-adic analysis \cite{E}, \cite{Rob}.
It is therefore natural to study dynamics generated by rational functions
in $p$-adic analysis.
In this paper we consider $(2,2)$-rational functions on the field of complex
$p$-adic numbers and study behavior of trajectories of the dynamical systems
generated by such functions.

The $p$-adic dynamical systems arise in Diophantine
geometry in the constructions of canonical heights, used for
counting rational points on algebraic varieties over a number
field, as in \cite{CS}. Moreover $p$-adic dynamical systems are
effective in computer science (straight line programs),
in numerical analysis and in simulations (pseudorandom numbers),
uniform distribution of sequences, cryptography
(stream ciphers, $T$-functions), combinatorics (Latin squares),
automata theory and formal languages, genetics. The monograph
\cite{AK09} contains the corresponding survey. For newer results see \cite{ARS}, \cite{AKY11}, \cite{Ana10}- \cite{Wal}.

Let us briefly mention papers which are devoted to dynamical systems
of $(n,m)$-rational functions (this is not a complete review of $p$-adic dynamical systems of rational functions).  A polynomial function can be considered as a $(n,0)$-rational function (see for example, \cite{FL11}). Therefore, we start from
 review of such functions.
The most studied discrete $p$-adic dynamical systems (iterations of maps) are the so-called
monomial systems.

In \cite{A2}, \cite{K0} the behavior of a $p$-adic dynamical system $f (x) = x^n$ in the fields
of $p$-adic numbers $Q_p$ and $\mathbb C_p$ were studied.

In \cite{AKK} the properties of the nonlinear $p$-adic dynamic system
 $f(x) = x^2 + c$ with a single parameter $c$ (i.e., a $(2,0)$-rational function) on the integer $p$-adic numbers $\mathbb Z_p$ are investigated. This
dynamic system illustrates possible brain behaviors during remembering.

In \cite{KN}, dynamical systems defined by the functions
$f_q(x) = x^n + q(x)$, where the perturbation $q(x)$ is a polynomial
whose coefficients have small $p$-adic absolute value, was
studied.

In \cite{F}, \cite{MM} the dynamical systems associated with the
function $f(x) = x^3 + ax^2$ on the set of $p$-adic numbers is studied.
More general form of this function, i.e., $f(x) = x^{2n+1} + ax^{n+1}$,
is considered in \cite{UF}.

Papers \cite{FF}, \cite{M1} (see also references therein) are devoted to $(1,1)$-rational $p$-adic dynamical systems.

In \cite{ARS} and \cite{KMa} the trajectories of an arbitrary $(2,1)$-rational
$p$-adic dynamical systems in a complex $p$-adic field $\C_p$ are studied.

The paper \cite{RS} is devoted to a
$(3,2)$-rational $p$-adic dynamical system in $\C_p$, when there exists a unique fixed point.

In \cite{S} we continued investigation of the $(3,2)$-rational $p$-adic dynamical systems in $\C_p$, when there are two fixed points.

In this paper we investigate behavior of trajectory of a
$(2,2)$-rational $p$-adic dynamical system in $\C_p$.

The paper is organized as follows: in
Section 2 we give some preliminaries. Section 3 contains the
definition of the $(2,2)$-rational function and main results about behavior of trajectories of
the $p$-adic dynamical system.  Siegel
disks of these dynamical systems are studied.
We obtain an upper bound for the set of limit points of each trajectory.
  We give explicit formulas of radiuses of spheres, with the center at the fixed point,
   containing some points such that
the trajectories of the points after a finite step
come to zeros of the denominator of the rational function.
In Section 4 for a class of $(2,2)$-rational functions
we study ergodicity properties of the dynamical systems on the set of $p$-adic numbers $Q_p$.
For each such function we describe all possible invariant spheres. We study 
ergodicity of each $p$-adic dynamical system with respect to Haar measure
reduced on each invariant sphere. It is proved that the dynamical 
systems are not ergodic.

\section{Preliminaries}

\subsection{$p$-adic numbers}

Let $Q$ be the field of rational numbers. The greatest common
divisor of the positive integers $n$ and $m$ is denotes by
$(n,m)$. Every rational number $x\neq 0$ can be represented in the
form $x=p^r\frac{n}{m}$, where $r,n\in\mathbb{Z}$, $m$ is a
positive integer, $(p,n)=1$, $(p,m)=1$ and $p$ is a fixed prime
number.

The $p$-adic norm of $x$ is given by
$$
|x|_p=\left\{
\begin{array}{ll}
p^{-r}, & \ \textrm{ for $x\neq 0$},\\[2mm]
0, &\ \textrm{ for $x=0$}.\\
\end{array}
\right.
$$
It has the following properties:

1) $|x|_p\geq 0$ and $|x|_p=0$ if and only if $x=0$,

2) $|xy|_p=|x|_p|y|_p$,

3) the strong triangle inequality
$$
|x+y|_p\leq\max\{|x|_p,|y|_p\},
$$

3.1) if $|x|_p\neq |y|_p$ then $|x+y|_p=\max\{|x|_p,|y|_p\}$,

3.2) if $|x|_p=|y|_p$ then $|x+y|_p\leq |x|_p$,

this is a non-Archimedean one.

The completion of $Q$ with  respect to $p$-adic norm defines the
$p$-adic field which is denoted by $Q_p$ (see \cite{Ko}).

The algebraic completion of $Q_p$ is denoted by $\C_p$ and it is
called {\it complex $p$-adic numbers}.  For any $a\in\C_p$ and
$r>0$ denote
$$
U_r(a)=\{x\in\C_p : |x-a|_p<r\},\ \ V_r(a)=\{x\in\C_p :
|x-a|_p\leq r\},
$$
$$
S_r(a)=\{x\in\C_p : |x-a|_p= r\}.
$$

A function $f:U_r(a)\to\C_p$ is said to be {\it analytic} if it
can be represented by
$$
f(x)=\sum_{n=0}^{\infty}f_n(x-a)^n, \ \ \ f_n\in \C_p,
$$ which converges uniformly on the ball $U_r(a)$.

\subsection{Dynamical systems in $\C_p$}

Recall some known facts concerning dynamical
systems $(f,U)$ in $\C_p$, where $f: x\in U\to f(x)\in U$ is an
analytic function and $U=U_r(a)$ or $\C_p$ (see for example \cite{PJS}).

Now let $f:U\to U$ be an analytic function. Denote
$f^n(x)=\underbrace{f\circ\dots\circ f}_n(x)$.

If $f(x_0)=x_0$ then $x_0$
is called a {\it fixed point}. The set of all fixed points of $f$
is denoted by Fix$(f)$. A fixed point $x_0$ is called an {\it
attractor} if there exists a neighborhood $U(x_0)$ of $x_0$ such
that for all points $x\in U(x_0)$ it holds
$\lim\limits_{n\to\infty}f^n(x)=x_0$. If $x_0$ is an attractor
then its {\it basin of attraction} is
$$
A(x_0)=\{x\in \C_p :\ f^n(x)\to x_0, \ n\to\infty\}.
$$
A fixed point $x_0$ is called {\it repeller} if there  exists a
neighborhood $U(x_0)$ of $x_0$ such that $|f(x)-x_0|_p>|x-x_0|_p$
for $x\in U(x_0)$, $x\neq x_0$.

Let $x_0$ be a fixed point of a
function $f(x)$.
Put $\lambda=f'(x_0)$. The point $x_0$ is attractive if $0<|\lambda|_p < 1$, {\it indifferent} if $|\lambda|_p = 1$,
and repelling if $|\lambda|_p > 1$.

The ball $U_r(x_0)$ (contained in $V$) is said to
be a {\it Siegel disk} if each sphere $S_{\r}(x_0)$, $\r<r$ is an
invariant sphere of $f(x)$, i.e. if $x\in S_{\r}(x_0)$ then all
iterated points $f^n(x)\in S_{\r}(x_0)$ for all $n=1,2\dots$.  The
union of all Siegel desks with the center at $x_0$ is said to {\it
a maximum Siegel disk} and is denoted by $SI(x_0)$.

\section{$(2,2)$-Rational $p$-adic dynamical systems}

In this paper we consider the dynamical system associated with the
$(2,2)$-rational function $f:\C_p\to\C_p$ defined by
\begin{equation}\label{fa}
f(x)=\frac{ax^2+bx+c}{x^2+dx+e}, \ \ a\neq 0,\ \ \ |b-ad|_p+|c-ae|_p\ne 0, \ \  a,b,c,d,e\in \C_p.
\end{equation}
where  $x\neq \hat x_{1,2}=\frac{-d\pm\sqrt{d^2-4e}}{2}$.

\begin{rk} We note that if $b=ad$ and $c=ae$ then from (\ref{fa})
we get $f(x)=a$, i.e., $f$ becomes a constant function.
Therefore we assumed $b\ne ad$ or $c\ne ae$.
\end{rk}

It is easy to see that for $(2,2)$-rational function (\ref{fa}) the equation
$f(x)=x$ for fixed points is equivalent to  the equation
\begin{equation}\label{ce}
x^3+(d-a)x^2+(e-b)x-c=0.
\end{equation}
Since $\C_p$ is algebraic close the equation (\ref{ce}) may have three
solutions with one of the following relations:

(i). One solution having multiplicity three;

(ii). Two solutions, one of which has multiplicity two;

(iii). Three distinct solutions.

In this paper we investigate the behavior of trajectories of an
arbitrary $(2,2)$-rational dynamical system in complex $p$-adic
filed $\C_p$ when the there is unique fixed point for $f$, i.e., we
consider the case (i).

The following lemma gives a criterion on parameters of
the function (\ref{fa}) guaranteing the uniqueness of its fixed point.

\begin{lemma}\label{tp} The function (\ref{fa}) has unique fixed point if and only if
\begin{equation}\label{fc}
   {{a-d}\over 3}=-\sqrt{{{e-b}\over 3}}=\sqrt[3]{c} \ \ or \ \ {{a-d}\over 3}=\sqrt{{{e-b}\over 3}}=\sqrt[3]{c}.
\end{equation}
\end{lemma}
\begin{proof} {\it Necessariness.} Assume (\ref{fa}) has a unique fixed point, say $x_0$.
Then the LHS of equation (\ref{ce}) (which is equivalent to $f(x)=x$) can be written as
$$x^3+(d-a)x^2+(e-b)x-c=(x-x_0)^3.$$
Consequently,
$$\left\{\begin{array}{lll}
3x_0=a-d \\[2mm]
3x_0^2=e-b \\[2mm]
x_0^3=c
\end{array}
\right.,$$ which gives
$$x_0={{a-d}\over 3}=\pm\sqrt{{{e-b}\over 3}}=\sqrt[3]{c}.$$

{\it Sufficientcy.} Assume the coefficients of (\ref{fa})
satisfy (\ref{fc}). Then it can be written as
\begin{equation}\label{fd}
f(x)=\frac{ax^2+bx+({{a-d}\over 3})^3}{x^2+dx+{(a-d)^2\over 3}+b},
\ \ a\neq0,\ \ b\ne -2a^2 \ \ \mbox{or} \ \ d\ne -2a, \ \  a,b,d\in \C_p.
\end{equation}
In this case the equation $f(x)=x$ can be written as
$$(x-{{a-d}\over 3})^3=0.$$
Thus $f(x)$ has unique fixed point $x_0={{a-d}\over 3}$.
\end{proof}

It follows from this lemma that if the function (\ref{fa}) has unique
fixed point then it has the form (\ref{fd}).
Thus we study the dynamical system $(f,\C_p)$ with $f$ given by (\ref{fd}).

For (\ref{fd}) we have
$$f'(x_0)=f'({{a-d}\over 3})=1,$$
i.e.,  the point $x_0$ is an indifferent point for (\ref{fd}).

In (\ref{fd}) one assumes $x^2+dx+{{(a-d)^2}\over 3}+b\neq 0$, i.e., $x\neq
x_{1,2}=-{d\over 2}\pm\sqrt{{d^2\over 4}-{{(a-d)^2}\over 3}-b}$.

 For any $x\in \C_p$, $x\ne x_{1,2}$, by simple calculations we
get
\begin{equation}\label{f2}
    |f(x)-x_0|_p=|x-x_0|_p\ \cdot{|{{2a+d}\over 3}
    (x-x_0)+(x_0-x_1)(x_0-x_2)|_p\over {|(x-x_0)+(x_0-x_1)|_p|(x-x_0)+(x_0-x_2)|_p}}.
\end{equation}

Denote
$$\mathcal P=\{x\in \C_p: \exists n\in \N\cup\{0\}, f^n(x)\in\{x_1, x_2\}\},$$
$$\delta=\left|{{2a+d}\over3}\right|_p, \ \ \alpha=|x_0-x_1|_p \ \ {\rm and} \ \ \beta=|x_0-x_2|_p.$$

\begin{lemma}\label{1}
\begin{itemize}
\item[1.] If $\left|\sqrt{{d^2\over 4}-{{(a-d)^2}\over
3}-b}\right|_p\ne\left|{{2a+d}\over 6}\right|_p$, then
$\alpha=\beta$.

\item[2.] If $\left|\sqrt{{d^2\over 4}-{{(a-d)^2}\over
3}-b}\right|_p=\left|{{2a+d}\over 6}\right|_p$, then
\end{itemize}
\begin{itemize}
\item $\alpha\leq\delta$ and $\beta\leq\delta$ for all
$p\geq3$.
\item $\alpha\leq2\delta$ and $\beta\leq2\delta$ for
$p=2$.
\end{itemize}
\end{lemma}
\begin{proof} This follows from properties of the norm $|\cdot|_p$.
\end{proof}

\begin{rk}\label{r2} It is easy to see that $x_0-x_1$ and $x_0-x_2$
are symmetric in (\ref{f2}), i.e., if we replace them then RHS of
(\ref{f2}) does not change. Therefore we consider the dynamical
system $(f, \C_p\setminus\mathcal P)$ for cases $\alpha=\beta$ and $\alpha<\beta$.
\end{rk}

\subsection{Case: $\alpha=\beta$}

Let us consider the following
functions:

For $\alpha>\delta$ define the function $\varphi_{\alpha,\delta}:
[0,+\infty)\to [0,+\infty)$ by
$$\varphi_{\alpha,\delta}(r)=\left\{\begin{array}{lllll}
r, \ \ {\rm if} \ \ r<\alpha\\[2mm]
\alpha^*, \ \ {\rm if} \ \ r=\alpha\\[2mm]
{{\alpha^2}\over r}, \ \ {\rm if} \ \ \alpha<r<{{\alpha^2}\over \delta}\\[2mm]
\delta^*, \ \ {\rm if} \ \ r={{\alpha^2}\over \delta}\\[2mm]
\delta, \ \ \ \ {\rm if} \ \ r>{{\alpha^2}\over \delta}
\end{array}
\right.
$$
where $\alpha^*$ and $\delta^*$ some positive numbers with
$\alpha^*\geq\alpha$, $\delta^*\leq \delta$.

For $\alpha<\delta$ define the function $\phi_{\alpha,\delta}:
[0,+\infty)\to [0,+\infty)$ by
$$\phi_{\alpha,\delta}(r)=\left\{\begin{array}{lllll}
r, \ \ {\rm if} \ \ r<{{\alpha^2}\over \delta}\\[2mm]
\alpha', \ \ {\rm if} \ \ r={{\alpha^2}\over \delta}\\[2mm]
{{\delta r^2}\over {\alpha^2}}, \ \ {\rm if} \ \ {{\alpha^2}\over \delta}<r<\alpha\\[2mm]
\delta', \ \ {\rm if} \ \ r=\alpha\\[2mm]
\delta, \ \ \ \ {\rm if} \ \ r>\alpha
\end{array}
\right.
$$
where $\alpha'$ and $\delta'$ some positive numbers with
$\alpha'\leq{{\alpha^2}\over \delta}$, $\delta'\geq \delta$.

For $\alpha=\delta$ define the function $\psi_{\alpha}:
[0,+\infty)\to [0,+\infty)$ by
$$\psi_{\alpha}(r)=\left\{\begin{array}{lll}
r, \ \ {\rm if} \ \ r<\alpha\\[2mm]
\hat\alpha, \ \ {\rm if} \ \ r=\alpha\\[2mm]
\alpha, \ \  {\rm if} \ \ r>\alpha
\end{array}
\right.
$$
where  $\hat\alpha$ some positive number.

Using the formula (\ref{f2}) we easily get the following:

\begin{lemma}\label{lf2} If $\alpha=\beta$ and
$x\in S_r(x_0)$, then the following formula
holds for function (\ref{fd})
$$|f^n(x)-x_0|_p=\left\{\begin{array}{lll}
\varphi_{\alpha,\delta}^n(r), \ \ \mbox{if} \ \ \alpha>\delta\\[2mm]
\phi_{\alpha,\delta}^n(r), \ \ \mbox{if} \ \ \alpha<\delta\\[2mm]
\psi_{\alpha}^n(r), \ \ \mbox{if} \ \ \alpha=\delta.
\end{array}\right.$$
\end{lemma}
Thus the $p$-adic dynamical system $f^n(x), n\geq 1, x\in
\C_p\setminus{\mathcal P}$ is related to the real dynamical
systems generated by $\varphi_{\alpha,\delta}$,
$\phi_{\alpha,\delta}$ and $\psi_{\alpha}$. Now we are going to
study these (real) dynamical systems.

\begin{lemma}\label{l1} If $\alpha>\delta$, then the dynamical
system generated by $\varphi_{\alpha,\delta}(r)$ has
the following properties:
\begin{itemize}
\item[1.] ${\rm Fix}(\varphi_{\alpha,\delta})=\{r: 0\leq r<\alpha\}\cup\{\alpha:\, if \,
\alpha^*=\alpha\}$.
\item[2.] If $r>\alpha$, then
$$\varphi_{\alpha,\delta}^n(r)=\left\{\begin{array}{lll}
{{\alpha^2}\over r}, \ \ \mbox{for all} \ \ \alpha<r<{{\alpha^2}\over \delta}\\[2mm]
\delta^*, \ \ \mbox{for} \ \ r={{\alpha^2}\over \delta}\\[2mm]
\delta, \ \ \ \ \mbox{for all} \ \ r>{{\alpha^2}\over \delta}
\end{array}
\right.$$ $\mbox{for any} \ \ n\geq 1$.
\item[3.] If $r=\alpha$ and $\alpha^*>\alpha$, then
$$\varphi_{\alpha,\delta}^n(r)=\left\{\begin{array}{lll}
{{\alpha^2}\over \alpha^*}, \ \ \mbox{if} \ \ \alpha<\alpha^*<{{\alpha^2}\over \delta}\\[2mm]
\delta^*, \ \ \mbox{if} \ \ \alpha^*={{\alpha^2}\over \delta}\\[2mm]
\delta, \ \ \ \ \mbox{if} \ \ \alpha^*>{{\alpha^2}\over \delta}
\end{array}
\right.$$ $\mbox{for any} \ \ n\geq 2$.
\end{itemize}
\end{lemma}

\begin{proof} 1. This is the result of a simple analysis
of the equation $\varphi_{\alpha,\delta}(r)=r$.

2. If $r>\delta$, then
$$\varphi_{\alpha,\delta}(r)=\left\{\begin{array}{lll}
{\alpha^2\over r}, \ \ {\rm if} \ \ \alpha<r<{{\alpha^2}\over \delta}\\[2mm]
\delta^*, \ \ {\rm if} \ \ r={{\alpha^2}\over \delta}\\[2mm]
\delta, \ \ {\rm if} \ \ r>{{\alpha^2}\over \delta}.
\end{array}
\right.
$$
Consequently,
$$\alpha<r<{{\alpha^2}\over \delta} \ \ \Rightarrow \ \ \delta<{{\alpha^2}\over r}<\alpha \ \ \Rightarrow \ \ \varphi_{\alpha,\delta}(r)<\alpha.$$
If $r\geq{{\alpha^2}\over \delta}$, then by
$\delta^*\leq\delta<\alpha$ we have
$\varphi_{\alpha,\delta}(r)<\alpha$.
Thus $\varphi_{\alpha,\delta}(\varphi_{\alpha,\delta}(r))=\varphi_{\alpha,\delta}(r)$,
i.e., $\varphi_{\alpha,\delta}(r)$ is a fixed point of
$\varphi_{\alpha,\delta}$ for any $r>\alpha$. Consequently, for each $n\geq 1$ we have
$$\varphi_{\alpha,\delta}^n(r)=\left\{\begin{array}{lll}
{\alpha^2\over r}, \ \ {\rm if} \ \ \alpha<r<{{\alpha^2}\over\delta}\\[2mm]
\delta^*, \ \ {\rm if} \ \ r={{\alpha^2}\over\delta}\\[2mm]
\delta, \ \ {\rm if} \ \ r>{{\alpha^2}\over\delta}.
\end{array}
\right.$$

3. The part 3 easily follows from the parts 1 and 2.
\end{proof}

\begin{lemma}\label{l2} If $\alpha<\delta$, then the dynamical
system generated by $\phi_{\alpha,\delta}(r)$ has the following properties:
 \begin{itemize}
\item[A.] ${\rm Fix}(\phi_{\alpha,\delta})=
\{r: 0\leq r<{{\alpha^2}\over \delta}\}\cup\{{{\alpha^2}\over \delta}:\, if \,
\alpha'={{\alpha^2}\over \delta}\}\cup\{\delta\}$.
\item[B.] If $r>{{\alpha^2}\over \delta}$, then
$$\lim_{n \to \infty}\phi_{\alpha,\delta}^n(r)=\delta.$$
\item[C.] If $r={{\alpha^2}\over \delta}$ and $\alpha'<{{\alpha^2}\over \delta}$, then
$\phi_{\alpha,\delta}^n(r)=\alpha'$ $\mbox{for all} \ \ n\geq 1$.
\end{itemize}
\end{lemma}
\begin{proof} A. This is the result of a simple analysis of the
equation $\phi_{\alpha,\delta}(r)=r$.

B. By definition of
$\phi_{\alpha,\delta}(r)$, for $r>\alpha$ we have
$\phi_{\alpha,\delta}(r)=\delta$, i.e., the function is constant.
Therefore
$$\lim_{n \to \infty}\phi_{\alpha,\delta}^n(r)=\delta.$$
For $r=\alpha$ we have $\phi_{\alpha,\delta}(\alpha)=\delta'\geq \delta$ and by condition
$\delta>\alpha$, we get
$\phi_{\alpha,\delta}(\alpha)>\alpha$. Consequently,
$$\lim_{n \to \infty}\phi_{\alpha,\delta}^n(\alpha)=\delta.$$
Assume now  ${{\alpha^2}\over \delta}<r<\alpha$ then
$\phi_{\alpha,\delta}(r)={{\delta r^2}\over {\alpha^2}}$,
$\phi'_{\alpha,\delta}(r)={{2\delta r}\over {\alpha^2}}>2$ and
$$\phi_{\alpha,\delta}(({{\alpha^2}\over
\delta},\alpha))=({{\alpha^2}\over
\delta},\delta)\cup\{\delta'\}.$$
Since  $\phi'_{\alpha,\delta}(r)>2$
for $r\in ({{\alpha^2}\over \delta},\alpha)$ there exists $n_0\in N$ such that
$\phi_{\alpha,\delta}^{n_0}(r)\in (\alpha,\delta)$.
Hence for $n\geq
n_0$ we get $\phi_{\alpha,\delta}^n(r)>\alpha$ and consequently
$$\lim_{n \to \infty}\phi_{\alpha,\delta}^n(r)=\delta.$$

C. If $r={{\alpha^2}\over \delta}$ and $\alpha'<{{\alpha^2}\over
\delta}$ then
$\phi_{\alpha,\delta}(r)=\alpha'<{{\alpha^2}\over \delta}$. Moreover,
$\alpha'$ is a fixed point for the function $\phi_{\alpha,\delta}$.
Thus for $n\geq 1$ we obtain
$\phi_{\alpha,\delta}^n(r)=\alpha'.$
\end{proof}

\begin{lemma}\label{l3} If $\alpha=\delta$, then the dynamical
system generated by $\psi_{\alpha}(r)$ has the following properties:
\begin{itemize}
\item[I.] ${\rm Fix}(\psi_{\alpha})=\{r: 0\leq r<\alpha\}\cup\{\alpha:\, if \,
\hat\alpha=\alpha\}$.
\item[II.] If $r>\alpha$, then $\psi_{\alpha}(r)=\alpha$.
\item[III.] Let $r=\alpha$.
\begin{itemize}
\item[III.i)] If $\hat\alpha<\alpha$, then
$\psi^{n}_{\alpha}(r)=\hat\alpha$, for any $n\geq 1$.
\item[III.ii)] If $\hat\alpha>\alpha$, then
$\psi^2_{\alpha}(\alpha)=\alpha.$
\end{itemize}
\end{itemize}
\end{lemma}

\begin{proof} I. This is the result of a simple
analysis of the equation $\psi_{\alpha}(r)=r$.

II. By definition of $\psi_{\alpha}(r)$, for any
$r>\alpha$ we have $\psi_{\alpha}(r)=\alpha$.

III. If $r=\alpha$ then $\psi_{\alpha}(r)=\hat\alpha$.

For $\hat\alpha\leq\alpha$ we have
$\psi_{\alpha}(\hat\alpha)=\hat\alpha$. Thus for all
$n\geq 1$ one has $\psi_{\alpha}^{n}(r)=\hat\alpha$.

In case $\hat\alpha>\alpha$ we have
$\psi_{\alpha}(\hat\alpha)=\alpha$,
$\psi_{\alpha}(\alpha)=\hat\alpha$. Hence
$\psi^2_{\alpha}(\alpha)=\alpha.$
\end{proof}

Now we shall apply these lemmas to study of the $p$-adic
dynamical system generated by function (\ref{fd}).

For $\alpha>\delta$ denote the following
$$\alpha^*(x)=|f(x)-x_0|_p, \ \ {\rm if} \ \ x\in
S_{\alpha}(x_0)$$ and
$$\delta^*(x)=|f(x)-x_0|_p, \ \ {\rm if} \ \ x\in
S_{{\alpha^2}\over \delta}(x_0).$$ Then using Lemma \ref{lf2} and
Lemma \ref{l1} we obtain the following

 \begin{thm}\label{t1} If $\alpha>\delta$, then
 the $p$-adic dynamical system generated by
 function (\ref{fd}) has the following properties:
\begin{itemize}
\item[1.]
\begin{itemize}
\item[1.1)] $SI(x_0)=U_{\alpha}(x_0)$.
\item[1.2)] $\mathcal P\subset S_{\alpha}(x_0)$.
\end{itemize}
\item[2.] If $r>\alpha$ and $x\in S_{r}(x_0)$,
then $$f^n(x)\in\left\{\begin{array}{lll}
S_{\alpha^2\over r}(x_0), \ \ \mbox{for all} \ \ \alpha<r<{{\alpha^2}\over \delta}\\[2mm]
S_{\delta^*(x)}(x_0), \ \ \mbox{for} \ \ r={{\alpha^2}\over \delta}\\[2mm]
S_{\delta}(x_0), \ \ \mbox{for all} \ \ r>{{\alpha^2}\over
\delta},
\end{array}
\right.$$ for any $n\geq 1$.
\item[3.] If $x\in S_{\alpha}(x_0)\setminus\mathcal P$,
then one of the following two possibilities holds:
\begin{itemize}
\item[3.1)]There exists $k\in N$ and $\mu_k>\alpha$
such that $f^k(x)\in S_{\mu_k}(x_0)$ and $$f^m(x)\in\left\{\begin{array}{lll}
S_{\alpha^2\over {\mu_k}}(x_0), \ \ \mbox{for all} \ \ \alpha<\mu_k<{{\alpha^2}\over \delta}\\[2mm]
S_{\delta^*(f^k(x))}(x_0), \ \ \mbox{for} \ \ \mu_k={{\alpha^2}\over \delta}\\[2mm]
S_{\delta}(x_0), \ \ \mbox{for all} \ \ \mu_k>{{\alpha^2}\over
\delta}
\end{array}
\right.$$ for any $m\geq k+1$ and $f^m(x)\in S_{\alpha}(x_0)$ if
$m\leq k-1$.
\item[3.2)] The trajectory $\{f^k(x), k\geq 1\}$ is a subset of
$S_{\alpha}(x_0)$.
\end{itemize}
\end{itemize}
\end{thm}
\begin{proof}
The part 2 easily follows from Lemma \ref{lf2} and  the part 2 of Lemma \ref{l1}.

3. Take $x\in S_\alpha(x_0)\setminus \mathcal P$ then we have
$$
|f(x)-x_0|_p={{\alpha^3}\over{\left|(x-x_0)+(x_0-x_1)\right|_p\left|(x-x_0)+(x_0-x_2)\right|_p}}\geq\alpha.
$$
If $|f(x)-x_0|_p>\alpha$ then there is $\mu_1>\alpha$ such that
$f(x)\in S_{\mu_1}(x_0)$ and by part 2 we have
$$f^m(x)\in\left\{\begin{array}{lll}
S_{\alpha^2\over {\mu_1}}(x_0), \ \ \mbox{for all} \ \ \alpha<\mu_1<{{\alpha^2}\over \delta}\\[2mm]
S_{\delta^*(f(x))}(x_0), \ \ \mbox{for} \ \ \mu_1={{\alpha^2}\over \delta}\\[2mm]
S_{\delta}(x_0), \ \ \mbox{for all} \ \ \mu_1>{{\alpha^2}\over
\delta}
\end{array}
\right.$$ for any $m\geq 2$. So in this case $k=1$.

If $|f(x)-x_0|_p=\alpha$ then we consider the following
$$
|f^2(x)-x_0|_p={{\alpha^3}\over{\left|(f(x)-x_0)+(x_0-x_1)\right|_p\left|(f(x)-x_0)+(x_0-x_2)\right|_p}}\geq\alpha.
$$
Now, if $|f^2(x)-x_0|_p>\alpha$ then there is $\mu_2>\alpha$ such
that $f^2(x)\in S_{\mu_2}(x_0)$ and by part 2 we get
$$f^m(x)\in\left\{\begin{array}{lll}
S_{\alpha^2\over {\mu_2}}(x_0), \ \ \mbox{for all} \ \ \alpha<\mu_2<{{\alpha^2}\over \delta}\\[2mm]
S_{\delta^*(f^2(x))}(x_0), \ \ \mbox{for} \ \ \mu_2={{\alpha^2}\over \delta}\\[2mm]
S_{\delta}(x_0), \ \ \mbox{for all} \ \ \mu_2>{{\alpha^2}\over
\delta}
\end{array}
\right.$$ for any $m\geq 3$. So in this case $k=2$.

If $|f^2(x)-x_0|_p=\alpha$ then we can continue the argument and
get the following inequality
$$|f^k(x)-x_0|_p\geq\alpha.$$
Hence in each step we may have two possibilities:
$|f^k(x)-x_0|_p=\alpha$ or $|f^k(x)-x_0|_p>\alpha$. In case
$|f^k(x)-x_0|_p>\alpha$ there exists $\mu_k$ such that $f^k(x)\in
S_{\mu_k}(x_0)$, and
$$f^m(x)\in\left\{\begin{array}{lll}
S_{\alpha^2\over {\mu_k}}(x_0), \ \ \mbox{for all} \ \ \alpha<\mu_k<{{\alpha^2}\over \delta}\\[2mm]
S_{\delta^*(f^k(x))}(x_0), \ \ \mbox{for} \ \ \mu_k={{\alpha^2}\over \delta}\\[2mm]
S_{\delta}(x_0), \ \ \mbox{for all} \ \ \mu_k>{{\alpha^2}\over
\delta}
\end{array}
\right.$$ for any $m\geq k+1$. If $|f^k(x)-x_0|_p=\alpha$ for any
$k\in \N$ then $\{f^k(x), k\geq 1\}\subset S_\alpha(x_0)$.

1. By parts 2 and 3 of theorem we know that $S_r(x_0)$ is not an invariant of $f$ for
$r\geq\alpha$. Consequently, $SI(x_0)\subset U_{\alpha}(x_0)$.

By Lemma \ref{lf2} and part 1 of  Lemma \ref{l1} if
$r<\alpha$ and $x\in S_r(x_0)$ then
$|f^n(x)-x_0|_p=\varphi^n_{\alpha,\delta}(r)=r$, i.e., $f^n(x)\in
S_r(x_0)$. Hence $U_{\alpha}(x_0)\subset SI(x_0)$ and thus
$SI(x_0)=U_{\alpha}(x_0).$

Since $|x_0-x_1|_p=|x_0-x_2|_p=\alpha$ we have $x_i\not\in
U_{\alpha}(x_0), \, i=1,2$. From $f(U_{\alpha}(x_0))\subset
U_{\alpha}(x_0)$ it follows that $$U_{\alpha}(x_0)\cap\mathcal P=\{x\in
U_{\alpha}(x_0): \exists n\in N\cup\{0\}, \, f^n(x)\in\{x_1, x_2\}
\}=\emptyset.$$
By part 2 of theorem for $r>\alpha$ we have
$f(S_r(x_0))\subset U_{\alpha}(x_0)$. Thus
$$(\C_p\setminus{V_{\alpha}(x_0)})\cap\mathcal P=\emptyset,$$
i.e., $\mathcal P\subset S_{\alpha}(x_0)$.
\end{proof}

 By Lemma \ref{lf2} and Lemma
\ref{l2} we get

\begin{thm}\label{t2} If $\alpha<\delta$, then
 the $p$-adic dynamical system generated by function (\ref{fd}) has the following properties:
\begin{itemize}
\item[A.]
\begin{itemize}
\item[A.a)] $SI(x_0)=U_{{{\alpha^2}\over \delta}}(x_0)$.
\item[A.b)] $f(S_{\delta}(x_0))\subset S_{\delta}(x_0)$, i.e., $S_{\delta}(x_0)$ is an invariant.
\end{itemize}
\item[B.] If $r>{{\alpha^2}\over \delta}$ and $x\in S_{r}(x_0)\setminus\mathcal P$, then
$$\lim_{n\to\infty}f^n(x)\in
S_{\delta}(x_0).$$
\item[C.] If $r={{\alpha^2}\over \delta}$, then one of the
following two possibilities holds:
\begin{itemize}
\item[C.a)] There exists $k\in N$ and
$\mu_k<{{\alpha^2}\over\delta}$ such that $f^m(x)\in
S_{\mu_k}(x_0)$ for any $m\geq k$ and $f^m(x)\in
S_{{\alpha^2}\over\delta}(x_0)$ if $m\leq k-1$.
\item[C.b)] The trajectory $\{f^k(x), k\geq 1\}$ is a subset of
$S_{{\alpha^2}\over\delta}(x_0)$.
\end{itemize}
\end{itemize}
\end{thm}
\begin{proof}
A. By Lemma \ref{lf2} and part A of Lemma \ref{l2} we see that
 spheres $S_r(x_0)$ and $S_{\delta}(x_0)$ are invariant for $f$ for any $r<{{\alpha^2}\over\delta}$. Thus $SI(x_0)=U_{{\alpha^2}\over\delta}(x_0)$.

B. Follows from Lemma \ref{lf2} and part B of Lemma
\ref{l2}.

C. If $x\in S_{{\alpha^2}\over\delta}(x_0)$ then we have
$$
|f(x)-x_0|_p={\left|{{2a+d}\over
3}(x-x_0)+(x_0-x_1)(x_0-x_2)\right|_p\over{\delta}}\leq{{\alpha^2}\over\delta}.
$$
If $|f(x)-x_0|_p<{{\alpha^2}\over\delta}$ then there is
$\mu_1<{{\alpha^2}\over\delta}$ such that $f^m(x)\in
S_{\mu_1}(x_0)$ for any $m\geq 1$ (see part A of Lemma \ref{l2}). So in this case $k=1$.

If $|f(x)-x_0|_p={{\alpha^2}\over\delta}$ then we consider the
following
$$
|f^2(x)-x_0|_p={\left|{{2a+d}\over
3}(f(x)-x_0)+(x_0-x_1)(x_0-x_2)\right|_p\over{\delta}}\leq{{\alpha^2}\over\delta}.
$$
Now, if $|f^2(x)-x_0|_p<{{\alpha^2}\over\delta}$ then there is
$\mu_2<{{\alpha^2}\over\delta}$ such that $f^m(x)\in
S_{\mu_2}(x_0)$ for any $m\geq 2$. So in this case $k=2$.

If $|f^2(x)-x_0|_p={{\alpha^2}\over\delta}$ then we can continue
the argument and get the following inequality
$$|f^k(x)-x_0|_p\leq{{\alpha^2}\over\delta}.$$
Hence in each step we may have two possibilities:
$|f^k(x)-x_0|_p={{\alpha^2}\over\delta}$ or
$|f^k(x)-x_0|_p<{{\alpha^2}\over\delta}$. In case
$|f^k(x)-x_0|_p<{{\alpha^2}\over\delta}$ there exists $\mu_k$ such
that $f^m(x)\in S_{\mu_k}(x_0)$  for any $m\geq k$. If
$|f^k(x)-x_0|_p={{\alpha^2}\over\delta}$ for any $k\in \N$ then
$\{f^k(x), k\geq 1\}\subset S_{{\alpha^2}\over\delta}(x_0)$.
\end{proof}

We note that $\mathcal P$ has the following form $$\mathcal
P=\bigcup_{k=0}^{\infty}{\mathcal P_k},  \ \ \mathcal P_k=\{x\in
\C_p: f^k(x)\in\{x_1, x_2\}\}.$$

\begin{thm}\label{tpk} If $\alpha<\delta$, then
\begin{itemize}
\item[1.] $\mathcal P_k\neq\emptyset, \ \ for \ \ any \ \ k=0,1,2,...
\,.$
\item[2.] $\mathcal P_k\subset S_{r_k}(x_0)$,   where
$r_k=\alpha\cdot\left(\alpha\over\delta\right)^{{2^k-1}\over{2^k}}$,
$k=0,1,2,... \,.$
\end{itemize}
\end{thm}
\begin{proof} 1. In case $k=0$ we have $\mathcal P_0=\{x_1, x_2\}\neq\emptyset$.

Assume for $k=n$ that $\mathcal P_n=\{x\in \C_p: f^n(x)\in\{x_1,
x_2\}\}\neq\emptyset$.

Now for $k=n+1$ to prove $\mathcal P_{n+1}=\{x\in \C_p: f^{n+1}(x)\in\{x_1,
x_2\}\}\neq\emptyset$ we have to show that the following equation
has at least one solution:
$$f^{n+1}(x)=x_i, \ \ \mbox{for some} \ \ i=1,2.$$

By our assumption  $\mathcal
P_n\neq\emptyset$ there exists $y\in\mathcal P_n$ such
that $f^n(y)\in \{x_1,x_2\}$. Now we show that there exists $x$ such that $f(x)=y$.
We note that the equation $f(x)=y$ can be written as
\begin{equation}\label{qe}
(a-y)x^2+(b-dy)x+\left({{a-d}\over 3}\right)^3-y\left[{{(a-d)^2}\over
3}+b\right]=0.
\end{equation}
We have  $|a-x_0|_p=\left|{{2a+d}\over 3}\right|_p=\delta$,
consequently, $a\in S_{\delta}(x_0)$.
Since $x_1,x_2\in S_\alpha(x_0)$ and by the part A.b)
of Theorem \ref{t2} we know that $S_{\delta}(x_0)$ is an invariant we
get $\mathcal P\cap S_\delta(x_0)=\emptyset$, for $\alpha<\delta$.
Thus $a\not\in\mathcal P$, consequently, $a-y\ne 0$.
Since $\C_p$ is algebraic closed the equation (\ref{qe})
has two solutions, say $x=t_1,t_2$. For $x\in \{t_1,t_2\}$ we get
$$f^{n+1}(x)=f^n(f(x))=f^n(y)\in \{x_1,x_2\}.$$ Hence
 $\mathcal P_{n+1}\ne\emptyset$. Therefore, by induction we get
$$\mathcal P_k\neq\emptyset, \ \ \mbox{for any} \ \ k=0,1,2,... \,.$$

2. We know $|x_0-x_1|_p=|x_0-x_2|_p=\alpha$. By condition $\alpha<\delta$, we get $\alpha>{{\alpha^2}\over\delta}$.
By (\ref{f2}) and part B of Lemma \ref{l2} for $x\in
S_{\alpha}(x_0)$, $x\neq x_{1,2}$ we have
$$\lim_{n\to\infty}f^n(x)\in S_{\delta}(x_0),$$
i.e., $S_{\alpha}(x_0)\cap\mathcal P=\{x_1, x_2\}=\mathcal
P_0$.
Denoting $r_0=\alpha$ we write $\mathcal P_0\subset S_{r_0}(x_0)$.
Now to find spheres containing the solutions of the equations $$f^k(x)=x_i, \ \ k=1,2,3,..., \ \ i=1,2.$$
We write the last equations in the form
$$f^k(x)-x_0=x_i-x_0, \ \ k=1,2,3,..., \ \ i=1,2.$$
For each  $k$ we want to find some $r_k$ such that the solution $x$ of $f^k(x)=x_i$, (for some $i=1,2$) belongs to $S_{r_k}(x_0)$, i.e.,  $x\in S_{r_k}(x_0)$. By Lemma \ref{lf2} we should have
$$\phi_{\alpha,\delta}^k(r_k)=\alpha.$$
Now if we show that the last equation has unique solution $r_k$ for each $k$, then
we get
$$\mathcal P_k=\{x\in \C_p: f^k(x)=x_i, i=1,2\}\subset
S_{r_k}(x_0).$$

By parts A and C of Lemma \ref{l2} we have
${{\alpha^2}\over\delta}<r_k\leq \alpha$.
Moreover, we have $r_0=\alpha$ and
${{\alpha^2}\over\delta}<r_k<\alpha$ for each $k=1,2,...$
For such $r_k$, by definition of $\phi_{\alpha,\delta}(r)$,  we have
$$\phi_{\alpha,\delta}(r_k)={{\delta r_k^2}\over{\alpha^2}}.$$
Thus $\phi_{\alpha,\delta}^k(r_k)=\alpha$ has the form
$$\phi_{\alpha,\delta}^k(r_k)={{\delta^{2^k-1}}\over{\alpha^{2(2^k-1)}}}r^{2^k}_k=\alpha$$
consequently
$$r^{2^k}_k=\alpha^{2^k}\cdot\left[\left({\alpha\over\delta}\right)^{{2^k-1}\over{2^k}}\right]^{2^k}.$$
Taking  $2^k$-root
we obtain unique positive solution:
$r_k=\alpha\cdot\left({\alpha\over\delta}\right)^{{2^k-1}\over{2^k}}$.
\end{proof}

If $\alpha=\delta$, then we denote
$$\hat\alpha(x)=|f(x)-x_0|_p, \
\ {\rm for} \ \ x\in S_{\alpha}(x_0).$$

Using Lemma \ref{lf2} and Lemma \ref{l3} we get

\begin{thm}\label{t3} If $\alpha=\delta$, then
 the $p$-adic dynamical system generated by function (\ref{fd}) has the following properties:
\begin{itemize}
\item[I.]
\begin{itemize}
\item[I.i)] $SI(x_0)=U_{\alpha}(x_0)$.
\item[I.ii)] $U_{\alpha}(x_0)\cap{\mathcal P}=\emptyset$.
\end{itemize}
\item[II.] If $r>\alpha$ and $x\in S_r(x_0)$, then $f(x)\in
S_{\alpha}(x_0)$.
\item[III.] Let $f^k(x)\in S_{\alpha}(x_0)\setminus\mathcal P$ for some $k=0,1,2,...$,
then
$$f^m(x)\in\left\{\begin{array}{lll}
S_{\hat\alpha(f^k(x))}(x_0), \ \ \mbox{if} \ \ \hat\alpha(f^k(x))\geq\alpha, \ \ m=k+1\\[2mm]
S_{\alpha}(x_0), \ \ \ \ \ \ \ \ \ \mbox{if} \ \ \hat\alpha(f^k(x))>\alpha, \ \ m=k+2\\[2mm]
S_{\hat\alpha(f^k(x))}(x_0), \ \ \mbox{if} \ \
\hat\alpha(f^k(x))<\alpha, \ \ \forall m\geq k+1.
\end{array}
\right.$$
\end{itemize}
\end{thm}
\begin{proof} Parts II-III of theorem easily follow from parts II-III of Lemma \ref{lf2}
and Lemma \ref{l3}.

I.  By the part I of Lemma \ref{lf2} and Lemma \ref{l3},
if $r<\alpha$ and $x\in S_r(x_0)$ then
$|f^n(x)-x_0|_p=\psi^n(r)=r$, i.e., for  $n\geq 1$ we have $f^n(x)\in
S_r(x_0)$. Consequently, $U_{\alpha}(x_0)\subset SI(x_0)$.

By the parts II-III of theorem we know that if $r\geq\alpha$
then $S_r(x_0)$ is not invariant for $f$. Hence $SI(x_0)\subset U_{\alpha}(x_0)$.
Therefore, $SI(x_0)= U_{\alpha}(x_0)$.

Since $|x_0-x_1|_p=|x_0-x_2|_p=\alpha$ we have $x_{1,2}\not\in
U_{\alpha}(x_0)$. Moreover, from $f(U_{\alpha}(x_0))\subset U_{\alpha}(x_0)$
we get
$$U_{\alpha}(x_0)\cap\mathcal P=\{x\in U_{\alpha}(x_0):
\exists n\in \N\cup\{0\}, f^n(x)\in\{x_1, x_2\} \}=\emptyset.$$
\end{proof}

\subsection{Case: $\alpha<\beta$.} In this case our arguments are similar to the ones used for the case $\alpha=\beta$, therefore we give results of this subsection without proofs.

Consider the following functions:

For $\delta<\alpha$ define the function
$\varphi_{\alpha,\beta,\delta}: [0,+\infty)\to [0,+\infty)$ by
$$\varphi_{\alpha,\beta,\delta}(r)=\left\{\begin{array}{lllllll}
r, \ \ {\rm if} \ \ r<\alpha\\[2mm]
\alpha^*, \ \ {\rm if} \ \ r=\alpha\\[2mm]
\alpha, \ \  {\rm if} \ \ \alpha<r<\beta\\[2mm]
\beta^*, \ \ {\rm if} \ \ r=\beta\\[2mm]
{{\alpha\beta}\over r}, \ \ {\rm if} \ \ \beta<r<{{\alpha\beta}\over \delta}\\[2mm]
\delta^*, \ \ {\rm if} \ \ r={{\alpha\beta}\over \delta}\\[2mm]
\delta, \ \ \ \ {\rm if} \ \ r>{{\alpha\beta}\over \delta}
\end{array}
\right.
$$
where $\alpha^*$, $\beta^*$ and $\delta^*$ some positive numbers
with $\alpha^*\geq\alpha$, $\beta^*\geq\alpha$ and $\delta^*\leq
\delta$.

For $\alpha=\delta$ define the function $\phi_{\alpha,\beta}:
[0,+\infty)\to [0,+\infty)$ by
$$\phi_{\alpha,\beta}(r)=\left\{\begin{array}{lllll}
r, \ \ {\rm if} \ \ r<\alpha\\[2mm]
\alpha', \ \ {\rm if} \ \ r=\alpha\\[2mm]
\alpha, \ \ {\rm if} \ \ \alpha<r<\beta\\[2mm]
\beta', \ \ {\rm if} \ \ r=\beta\\[2mm]
\alpha, \ \ \ \ {\rm if} \ \ r>\beta
\end{array}
\right.
$$
where $\alpha'$ and $\beta'$ some positive numbers with
$\alpha'\geq\alpha$, $\beta'>0$.

For $\alpha<\delta<\beta$ define the function
$\psi_{\alpha,\beta,\delta}: [0,+\infty)\to [0,+\infty)$ by
$$\psi_{\alpha,\beta,\delta}(r)=\left\{\begin{array}{lllllll}
r, \ \ {\rm if} \ \ r<\alpha\\[2mm]
\hat\alpha, \ \ {\rm if} \ \ r=\alpha\\[2mm]
\alpha, \ \  {\rm if} \ \ \alpha<r<{{\alpha\beta}\over\delta}\\[2mm]
\hat\delta, \ \ {\rm if} \ \ r={{\alpha\beta}\over\delta}\\[2mm]
{\delta\over\beta}r, \ \ {\rm if} \ \ {{\alpha\beta}\over\delta}<r<\beta\\[2mm]
\hat\beta, \ \ {\rm if} \ \ r=\beta\\[2mm]
\delta, \ \ {\rm if} \ \ r>\beta
\end{array}
\right.
$$
where $\hat\alpha$, $\hat\beta$ and $\hat\delta$ some positive
numbers with $\hat\alpha\geq\alpha$, $\hat\beta\geq\delta$ and
$\hat\delta\leq\alpha$.

For $\delta=\beta$ define the function $\eta_{\alpha,\beta}:
[0,+\infty)\to [0,+\infty)$ by
$$\eta_{\alpha,\beta}(r)=\left\{\begin{array}{lllll}
r, \ \ {\rm if} \ \ r<\alpha\\[2mm]
\bar\alpha, \ \ {\rm if} \ \ r=\alpha\\[2mm]
r, \ \ {\rm if} \ \ \alpha<r<\beta\\[2mm]
\bar\beta, \ \ {\rm if} \ \ r=\beta\\[2mm]
\beta, \ \ \ \ {\rm if} \ \ r>\beta
\end{array}
\right.
$$
where $\bar\alpha$ and $\bar\beta$ some positive numbers with
$\bar\alpha>0$, $\bar\beta\geq\beta$.

For $\beta<\delta$ define the function
$\zeta_{\alpha,\beta,\delta}: [0,+\infty)\to [0,+\infty)$ by
$$\zeta_{\alpha,\beta,\delta}(r)=\left\{\begin{array}{lllllll}
r, \ \ {\rm if} \ \ r<{{\alpha\beta}\over\delta}\\[2mm]
\tilde\delta, \ \ {\rm if} \ \ r={{\alpha\beta}\over\delta}\\[2mm]
{{\delta r^2}\over{\alpha\beta}}, \ \  {\rm if} \ \ {{\alpha\beta}\over\delta}<r<\alpha\\[2mm]
\tilde\alpha, \ \ {\rm if} \ \ r=\alpha\\[2mm]
{\delta\over\beta}r, \ \ {\rm if} \ \ \alpha<r<\beta\\[2mm]
\tilde\beta, \ \ {\rm if} \ \ r=\beta\\[2mm]
\delta, \ \ {\rm if} \ \ r>\beta
\end{array}
\right.
$$
where $\tilde\alpha$, $\tilde\beta$ and $\tilde\delta$ some
positive numbers with
$\tilde\alpha\geq{{\alpha\delta}\over\beta}$,
$\tilde\beta\geq\delta$ and
$\tilde\delta\leq{{\alpha\beta}\over\delta}$.

Using the formula (\ref{f2}) we easily get the following:

\begin{lemma}\label{lf22} If $\alpha<\beta$ and $x\in S_r(x_0)\setminus\mathcal P$, then the following formula
holds for function (\ref{fd})
$$|f^n(x)-x_0|_p=\left\{\begin{array}{lllll}
\varphi_{\alpha,\beta,\delta}^n(r), \ \ \mbox{if} \ \ \delta<\alpha\\[2mm]
\phi_{\alpha,\beta}^n(r), \ \  \ \ \mbox{if} \ \ \alpha=\delta\\[2mm]
\psi_{\alpha,\beta,\delta}^n(r), \ \ \mbox{if} \ \ \alpha<\delta<\beta\\[2mm]
\eta_{\alpha,\beta}^n(r), \ \ \ \ \mbox{if} \ \ \delta=\beta\\[2mm]
\zeta_{\alpha,\beta,\delta}^n(r), \ \ \mbox{if} \ \ \beta<\delta
\end{array}\right.$$
\end{lemma}

Thus the $p$-adic dynamical system $f^n(x), n\geq 1, x\in
\C_p\setminus{\mathcal P}$ is related to the real dynamical
systems generated by $\varphi_{\alpha,\beta,\delta}$,
$\phi_{\alpha,\beta}$, $\psi_{\alpha,\beta,\delta}$,
$\eta_{\alpha,\beta}$ and $\zeta_{\alpha,\beta,\delta}$.

The following simple lemmas are devoted to properties of these (real) dynamical systems.

\begin{lemma}\label{lf221} If $\delta<\alpha$, then the dynamical system generated by $\varphi_{\alpha,\beta,\delta}(r)$ has the following properties:
\begin{itemize}
\item[1.] ${\rm Fix}(\varphi_{\alpha,\beta,\delta})=\{r: 0\leq r<\alpha\}\cup\{\alpha:\, if \,
\alpha^*=\alpha\}\cup\{\beta:\, if \, \beta^*=\beta\}$.
\item[2.] If $\alpha<r<\beta$, then $\varphi_{\alpha,\beta,\delta}(r)=\alpha.$
\item[3.] If $r>\beta$, then
$$\varphi_{\alpha,\beta,\delta}^n(r)=\left\{\begin{array}{lll}
{{\alpha\beta}\over r}, \ \ \mbox{for all} \ \ \beta<r<{{\alpha\beta}\over \delta}\\[2mm]
\delta^*, \ \ \mbox{for} \ \ r={{\alpha\beta}\over \delta}\\[2mm]
\delta, \ \ \ \ \mbox{for all} \ \ r>{{\alpha\beta}\over \delta}
\end{array}
\right.$$ $\mbox{for any} \ \ n\geq 1$.
\item[4.] Let $r=\alpha$.
\begin{itemize}
\item[4.1)] If $\alpha<\alpha^*<\beta$, then
$\varphi_{\alpha,\beta,\delta}^2(\alpha)=\alpha$.
\item[4.2)] If $\alpha^*=\beta$, then $\varphi_{\alpha,\beta,\delta}(\alpha)=\beta$.
\item[4.3)] If $\alpha^*>\beta$, then
$$\varphi_{\alpha,\beta,\delta}^n(\alpha)=\left\{\begin{array}{lll}
{{\alpha\beta}\over \alpha^*}, \ \ \mbox{if} \ \ \beta<\alpha^*<{{\alpha\beta}\over \delta}\\[2mm]
\delta^*, \ \ \mbox{if} \ \ \alpha^*={{\alpha\beta}\over \delta}\\[2mm]
\delta, \ \ \ \ \mbox{if} \ \ \alpha^*>{{\alpha\beta}\over \delta}
\end{array}
\right.$$ $\mbox{for any} \ \ n\geq 2$.
\end{itemize}
\item[5.] Let $r=\beta$.
\begin{itemize}
\item[5.1)] If $\alpha<\beta^*<\beta$, then
$\varphi_{\alpha,\beta,\delta}^2(\beta)=\alpha$.
\item[5.2)] If $\beta^*>\beta$, then
$$\varphi_{\alpha,\beta,\delta}^n(\beta)=\left\{\begin{array}{lll}
{{\alpha\beta}\over \beta^*}, \ \ \mbox{if} \ \ \beta<\beta^*<{{\alpha\beta}\over \delta}\\[2mm]
\delta^*, \ \ \mbox{if} \ \ \beta^*={{\alpha\beta}\over \delta}\\[2mm]
\delta, \ \ \ \ \mbox{if} \ \ \beta^*>{{\alpha\beta}\over \delta}
\end{array}
\right.$$ $\mbox{for any} \ \ n\geq 2$.
\end{itemize}
\end{itemize}
\end{lemma}

\begin{lemma}\label{lf222} If $\alpha=\delta$, then the
dynamical system generated by $\phi_{\alpha,\beta}(r)$ has the following properties:
\begin{itemize}
\item[1.] ${\rm Fix}(\phi_{\alpha,\beta})=\{r: 0\leq r<\alpha\}\cup\{\alpha:\, if \,
\alpha'=\alpha\}\cup\{\beta: \, if \, \beta'=\beta\}$.
\item[2.] If $r>\alpha$ and $r\neq\beta$, then
$\phi_{\alpha,\beta}(r)=\alpha.$
\item[3.] Let $r=\alpha$ and $\alpha'>\alpha$.
\begin{itemize}
\item[3.1)] If $\alpha'\neq\beta$, then $\phi_{\alpha,\beta}^2(\alpha)=\alpha.$
\item[3.2)] If $\alpha'=\beta$, then $\phi_{\alpha,\beta}(\alpha)=\beta.$
\end{itemize}
\item[4.] If $r=\beta$
\begin{itemize}
\item[4.1)] If $\beta'<\alpha$, then
$\phi_{\alpha,\beta}^n(\beta)=\beta'$ for all $n\geq 1$.
\item[4.2)] If $\beta'=\alpha$, then $\phi_{\alpha,\beta}(\beta)=\alpha$.
\item[4.3)] If $\beta'>\alpha$ and $\beta'\neq\beta$, then
$\phi_{\alpha,\beta}^2(\beta)=\alpha.$
\end{itemize}
\end{itemize}
\end{lemma}

\begin{lemma}\label{lf223} If $\alpha<\delta<\beta$, then the dynamical system generated by $\psi_{\alpha,\beta,\delta}(r)$ has the following properties:
\begin{itemize}
\item[1.] ${\rm Fix}(\psi_{\alpha,\beta,\delta})=\{r: 0\leq r<\alpha\}\cup\{\alpha:\, if \,
\hat\alpha=\alpha\}\cup\{\beta: \, if \, \hat\beta=\beta\}$.
\item[2.] If $\hat\alpha\not\in\{\alpha,\beta\}$ and $\hat\beta\neq\beta$, then there exists $n\in N$ such that $\psi^n_{\alpha,\beta,\delta}(r)=\alpha$
for all $r\geq\alpha$.
\item[3.] If $r=\alpha$ and $\hat\alpha=\beta$, then $\psi_{\alpha,\beta,\delta}(\alpha)=\beta$.
\end{itemize}
\end{lemma}

\begin{lemma}\label{lf224} If $\delta=\beta$, then the dynamical system generated by $\eta_{\alpha,\beta}(r)$ has the following properties:
\begin{itemize}
\item[1.] ${\rm Fix}(\eta_{\alpha,\beta})=\{r: 0\leq r<\alpha\}\cup\{\alpha:\, if \,
\bar\alpha=\alpha\}\cup\{r: \alpha<r<\beta\}\cup\{\beta:\, if \,
\bar\beta=\beta\}$.
\item[2.] If $r>\beta$, then $\eta_{\alpha,\beta}(r)=\beta.$
\item[3.] Let $r=\alpha$.
\begin{itemize}
\item[3.1)] If $\bar\alpha\neq\alpha$ and $\bar\alpha<\beta$, then
$\eta^n_{\alpha,\beta}(\alpha)=\bar\alpha$, for all $n\geq 1$.
\item[3.2)] If $\bar\alpha=\beta$, then $\eta_{\alpha,\beta}(\alpha)=\beta.$
\item[3.3)] If $\bar\alpha>\beta$, then $\eta^2_{\alpha,\beta}(\alpha)=\beta.$
\end{itemize}
\item[4.] If $r=\beta$ and $\bar\beta>\beta$, then $\eta^2_{\alpha,\beta}(\beta)=\beta.$
\end{itemize}
\end{lemma}

\begin{lemma}\label{lf225} If $\delta>\beta$, then the dynamical system generated by $\zeta_{\alpha,\beta,\delta}(r)$ has the following properties:
\begin{itemize}
\item[1.] ${\rm Fix}(\zeta_{\alpha,\beta,\delta})=\{r: 0\leq r<{{\alpha\beta}\over\delta}\}\cup\{{{\alpha\beta}\over\delta}:\, if \,
\tilde\delta={{\alpha\beta}\over\delta}\}\cup\{\delta\}$.
\item[2.] If $r>{{\alpha\beta}\over\delta}$, then
$$\lim_{n\to\infty}\zeta^n_{\alpha,\beta,\delta}(r)=\delta.$$
\item[3.] If $r={{\alpha\beta}\over\delta}$ and $\tilde\delta<{{\alpha\beta}\over\delta}$, then
$$\zeta^n_{\alpha,\beta,\delta}(r)=\tilde\delta$$ for any $n\geq 1$.
\end{itemize}
\end{lemma}

If $\alpha<\beta$, then we note that $\mathcal P$ has the
following form $\mathcal P=\mathcal P_{\alpha}\cup\mathcal
P_{\beta}$, where $$\mathcal P_{\alpha}=\{x\in \C_p: \exists n\in
\N\cup\{0\}, f^n(x)=x_1\} \ \ {\rm and} \ \ \mathcal
P_{\beta}=\{x\in \C_p: \exists n\in \N\cup\{0\}, f^n(x)=x_2\}.$$
By definitions of
$\varphi_{\alpha,\beta,\delta}$,  $\phi_{\alpha,\beta}$,
$\psi_{\alpha,\beta,\delta}$,  $\eta_{\alpha,\beta}$ and
$\zeta_{\alpha,\beta,\delta}$ and Lemma \ref{lf221}
- Lemma \ref{lf225} we have the following

\begin{thm}\label{lpp} If $\alpha<\beta$, then
\begin{itemize}
\item[1.] If $\delta<\alpha$, then $$\mathcal P_{\alpha}\subset\bigcup_{\alpha\leq r\leq\beta}S_{r}(x_0) \ \ {\rm and} \ \ \mathcal P_{\beta}\subset S_{\beta}(x_0).$$
\item[2.] If $\alpha\leq\delta<\beta$, then $$\mathcal P_{\alpha}\subset\bigcup_{r\geq\alpha}S_{r}(x_0) \ \ {\rm and} \ \ \mathcal P_{\beta}\subset S_{\beta}(x_0).$$
\item[3.] If $\delta=\beta$, then $$\mathcal P_{\alpha}\subset S_{\alpha}(x_0) \ \ {\rm and} \ \ \mathcal P_{\beta}\subset\bigcup_{r\geq\beta} S_{r}(x_0).$$
\item[4.] If $\delta>\beta$, then there are two sequences $\{r_k\}\subset({{\alpha\beta}\over\delta},\alpha]$ and $\{\r_k\}\subset({{\alpha\beta}\over\delta},\beta]$
such that
$$\mathcal P_{\alpha}\subset\bigcup^{\infty}_{k=1}S_{r_k}(x_0) \ \
{\rm and} \ \ \mathcal P_{\beta}\subset
\bigcup^{\infty}_{k=1}S_{\r_k}(x_0).$$
\end{itemize}
\end{thm}

Now we shall apply above-mentioned results for study of the $p$-adic
dynamical system generated by function (\ref{fd}).

For $\alpha<\beta$ denote the following
$$\alpha^*(x)=|f(x)-x_0|_p, \ \ {\rm if} \ \ x\in
S_{\alpha}(x_0)\setminus\{x_1\};$$ $$ \beta^*(x)=|f(x)-x_0|_p, \ \
{\rm if} \ \ x\in S_{\beta}(x_0)\setminus\{x_2\};$$
$$\delta^*(x)=|f(x)-x_0|_p, \ \ {\rm if} \ \ x\in
S_{{\alpha\beta}\over \delta}(x_0).$$

Then using Lemma \ref{lf22} and Lemma \ref{lf221} we obtain the
following

 \begin{thm}\label{tf221} If $\delta<\alpha<\beta$ and $x\in S_r(x_0)\setminus\mathcal P$, then
 the $p$-adic dynamical system generated by function (\ref{fd}) has the following properties:
\begin{itemize}
\item[1.] $SI(x_0)=U_{\alpha}(x_0)$.
\item[2.] If $\alpha<r<\beta$, then $f(x)\in S_{\alpha}(x_0)$.
\item[3.] Let $r>\beta$, then $$f^n(x)\in\left\{\begin{array}{lll}
S_{\alpha\beta\over r}(x_0), \ \ \mbox{for all} \ \ \alpha<r<{{\alpha\beta}\over \delta}\\[2mm]
S_{\delta^*(x)}(x_0), \ \ \mbox{for} \ \ r={{\alpha\beta}\over \delta}\\[2mm]
S_{\delta}(x_0), \ \ \mbox{for all} \ \ r>{{\alpha\beta}\over
\delta},
\end{array}
\right.$$ for any $n\geq 1$.
\item[4.] Let $x\in S_{\alpha}(x_0)\setminus\mathcal P$.
\begin{itemize}
\item[4.1)] If $\alpha^*(x)=\alpha$, then $f(x)\in
S_{\alpha}(x_0)$.
\item[4.2)] If $\alpha<\alpha^*(x)<\beta$, then $f^2(x)\in
S_{\alpha}(x_0)$.
\item[4.3)] If $\alpha^*(x)=\beta$, then $f(x)\in
S_{\beta}(x_0)$.
\item[4.4)] If $\alpha^*(x)>\beta$, then $$f^n(x)\in\left\{\begin{array}{lll}
S_{\alpha\beta\over{\alpha^*(x)}}(x_0), \ \ \mbox{for all} \ \ \alpha<{\alpha^*(x)}<{{\alpha\beta}\over \delta}\\[2mm]
S_{\delta^*(f(x))}(x_0), \ \ \mbox{for} \ \ {\alpha^*(x)}={{\alpha\beta}\over \delta}\\[2mm]
S_{\delta}(x_0), \ \ \mbox{for all} \ \
{\alpha^*(x)}>{{\alpha\beta}\over \delta}
\end{array}
\right.$$ for any $n\geq 2$.
\end{itemize}
\item[5.] Let $x\in S_{\beta}(x_0)\setminus\mathcal P$.
\begin{itemize}
\item[5.1)] If $\beta^*(x)=\alpha$, then $f(x)\in
S_{\alpha}(x_0)$.
\item[5.2)] If $\alpha<\beta^*(x)<\beta$, then $f^2(x)\in
S_{\alpha}(x_0)$.
\item[5.3)] If $\beta^*(x)=\beta$, then $f(x)\in
S_{\beta}(x_0)$.
\item[5.4)] If $\beta^*(x)>\beta$, then $$f^n(x)\in\left\{\begin{array}{lll}
S_{\alpha\beta\over{\beta^*(x)}}(x_0), \ \ \mbox{for all} \ \ \alpha<{\beta^*(x)}<{{\alpha\beta}\over \delta}\\[2mm]
S_{\delta^*(f(x))}(x_0), \ \ \mbox{for} \ \ {\beta^*(x)}={{\alpha\beta}\over \delta}\\[2mm]
S_{\delta}(x_0), \ \ \mbox{for all} \ \
{\beta^*(x)}>{{\alpha\beta}\over \delta}
\end{array}
\right.$$ for any $n\geq 2$.
\end{itemize}
\end{itemize}
\end{thm}

By Lemma \ref{lf22} and Lemma \ref{lf222} we get

\begin{thm}\label{tf222} If $\alpha=\delta$ and $x\in S_r(x_0)\setminus\mathcal P$, then
 the $p$-adic dynamical system generated by function (\ref{fd}) has the following properties:
\begin{itemize}
\item[1.] $SI(x_0)=U_{\alpha}(x_0)$.
\item[2.] If $r>\alpha$ and $r\ne\beta$, then
$f(x)\in S_{\alpha}(x_0)$.
\item[3.] Let $x\in S_{\alpha}(x_0)\setminus\mathcal P$.
\begin{itemize}
\item[3.1)] If $\alpha^*(x)=\alpha$, then $f(x)\in S_{\alpha}(x_0)$.
\item[3.2)] If $\alpha^*(x)>\alpha$ and $\alpha^*(x)\ne\beta$, then $f^2(x)\in S_{\alpha}(x_0)$.
\item[3.3)] If $\alpha^*(x)=\beta$, then $f(x)\in S_{\beta}(x_0)$.
\end{itemize}
\item[4.] Let $x\in S_{\beta}(x_0)\setminus\mathcal P$.
\begin{itemize}
\item[4.1)] If $\beta^*(x)<\alpha$, then $f^n(x)\in S_{\beta^*(x)}(x_0)$ for any $n\geq 1$.
\item[4.2)] If $\beta^*(x)=\alpha$, then $f(x)\in S_{\alpha}(x_0)$.
\item[4.3)] If $\beta^*(x)>\alpha$ and $\beta^*(x)\ne\beta$, then $f^2(x)\in S_{\alpha}(x_0)$.
\item[4.4)] If $\beta^*(x)=\beta$, then $f(x)\in S_{\beta}(x_0)$.
\end{itemize}
\end{itemize}
\end{thm}

Using the Lemma \ref{lf22} and Lemma \ref{lf223} we get

\begin{thm}\label{tf223} If $\alpha<\delta<\beta$ and $x\in S_r(x_0)\setminus\mathcal P$, then
 the $p$-adic dynamical system generated by function (\ref{fd}) has the following properties:
\begin{itemize}
\item[1.] $SI(x_0)=U_{\alpha}(x_0)$.
\item[2.] If $r>\alpha$ and $r\ne\beta$, then there exists $n\in N$ such that $f^n(x)\in
S_{\alpha}(x_0)$.
\item[3.] Let $r=\alpha$ ($r=\beta$).
\begin{itemize}
\item[3.1)] If $\alpha^*(x)\ne\beta$ ($\beta^*(x)\ne\beta$), then there exists $n\in N$ such that $f^n(x)\in
S_{\alpha}(x_0)$.
\item[3.2)] If $\alpha^*(x)=\beta$ ($\beta^*(x)=\beta$), then $f(x)\in
S_{\beta}(x_0)$.
\end{itemize}
\end{itemize}
\end{thm}

Using the Lemma \ref{lf22} and Lemma \ref{lf224} we get
 \begin{thm}\label{tf224} If $\delta=\beta$ and $x\in S_r(x_0)\setminus\mathcal P$, then
 the $p$-adic dynamical system generated by function (\ref{fd}) has the following properties:
\begin{itemize}
\item[1.]
\begin{itemize}
\item[1.1)] $SI(x_0)=U_{\alpha}(x_0)$.
\item[1.2)] If $\alpha<r<\beta$, then $S_r(x_0)$ is an invariant for $f$.
\end{itemize}
\item[2.] Let $r=\alpha$ and $x\in S_{\alpha}(x_0)\setminus\mathcal P$.
\begin{itemize}
\item[2.1)] If $\alpha^*(x)=\alpha$, then $f(x)\in S_{\alpha}(x_0)$.
\item[2.2)] If $\alpha^*(x)<\beta$ and $\alpha^*(x)\ne\alpha$, then $f^n(x)\in S_{\alpha^*(x)}(x_0)$ for any $n\geq 1$.
\item[2.3)] If $\alpha^*(x)=\beta$, then $f(x)\in S_{\beta}(x_0)$.
\item[2.4)] If $\alpha^*(x)>\beta$, then $f^2(x)\in S_{\beta}(x_0)$.
\end{itemize}
\item[3.] Let $r=\beta$ and $x\in S_{\beta}(x_0)\setminus\mathcal P$.
\begin{itemize}
\item[3.1)] If $\beta^*(x)=\beta$, then $f(x)\in S_{\beta}(x_0)$.
\item[3.2)] If $\beta^*(x)>\beta$, then $f^2(x)\in S_{\beta}(x_0)$.
\end{itemize}
\item[4.] If $r>\beta$, then $f(S_r(x_0))\subset S_{\beta}(x_0)$.
\end{itemize}
\end{thm}

By Lemma \ref{lf22} and Lemma \ref{lf225} we get

\begin{thm}\label{tf225} If $\delta>\beta$  and $x\in S_r(x_0)\setminus\mathcal P$, then
 the $p$-adic dynamical system generated by function (\ref{fd}) has the following properties:
\begin{itemize}
\item[1.]
\begin{itemize}
\item[1.1)] $SI(x_0)=U_{{{\alpha\beta}\over \delta}}(x_0)$.
\item[1.2)] The sphere $S_{\delta}(x_0)$ is invariant for $f$.
\end{itemize}
\item[2.] If $r>{{\alpha\beta}\over \delta}$, then
$$\lim_{n\to\infty}f^n(x)\in S_{\delta}(x_0).$$
\item[3.] If $r={{\alpha\beta}\over \delta}$, then one of the
following two possibilities holds:
\begin{itemize}
\item[3.1)] There exists $k\in N$ and
$\mu_k<{{\alpha\beta}\over\delta}$ such that $$f^m(x)\in
S_{\mu_k}(x_0)$$ for any $m\geq k$ and $f^m(x)\in
S_{{\alpha\beta}\over\delta}(x_0)$ if $m\leq k-1$.
\item[3.2)] The trajectory $\{f^k(x), k\geq 1\}$ is a subset of
$S_{{\alpha\beta}\over\delta}(x_0)$.
\end{itemize}
\end{itemize}
\end{thm}

\begin{rk} If $\alpha<\beta$ and $p\geq 3$ then by Lemma \ref{1} we have only
Theorem \ref{tf224} and Theorem \ref{tf225}. See Remark \ref{r2} for the case $\alpha>\beta$.
\end{rk}

\section{Dynamical system $f(x)={{ax^2+bx}\over{x^2+ax+b}}$ in $Q_p$ is not ergodic}

In this section we assume $a=d$ in (\ref{fd}), and suppose
that the square root $\sqrt{a^2-4b}$ exists in $Q_p$. Then (\ref{fd}) has the form
\begin{equation}\label{fab}
f(x)=\frac{ax^2+bx}{x^2+ax+b}
\end{equation}
where  $x\neq \hat x_{1,2}=\frac{-a\pm\sqrt{a^2-4b}}{2}$.

Consider the dynamical system (\ref{fab}) in
$Q_p$. It is easy to see that $x_0=0$ is unique fixed point for
(\ref{fab}).

We have $\delta=|a|_p$,  $\alpha=|\hat x_1|_p$,  $\beta=|\hat
x_2|_p$   and  $|b|_p=\alpha\beta$.

Since $\hat x_1+\hat x_2=-a$ we have $\delta\leq\max\{\alpha,\beta\}$.

Define the following sets
$$A_1=\{r: \, 0\leq r<\alpha\} \ \ {\rm if} \ \ \delta\leq\alpha=\beta;$$
$$A_2=\{r: \, r\in[0,\beta)\setminus\{\alpha\}\} \ \ {\rm if} \ \ \alpha<\beta=\delta;$$
and we denote $A=A_1\cup A_2$ for $\alpha\leq\beta$.

From previous sections we have the following
\begin{cor} The sphere $S_r(0)$ is invariant for $f$ if and only if $r\in A$.
\end{cor}
Since $x_0=0$ is an
indifferent fixed point, in this section we are
interested to study ergodicity properties of the dynamical system.

\begin{lemma}{\label{ab}}
For every closed ball $V_{\r}(c)\subset S_r(0), \, r\in A$ the
following equality holds $$f(V_{\r}(c))=V_{\r}(f(c)).$$
\end{lemma}
\begin{proof}
From inclusion $V_{\r}(c)\subset S_r(0)$ we have $|c|_p=r$.

Let $x\in V_{\r}(c)$, i.e. $|x-c|_p\leq\r$, then
\begin{equation}{\label{ab1}}
|f(x)-f(c)|_p=|x-c|_p\cdot\frac{|a^2cx+ab(x+c)-bcx+b^2|_p}{|(x-\hat
x_1)(x-\hat x_2)(c-\hat x_1)(c-\hat x_2)|_p}.
\end{equation}
We have $$|a^2cx|_p=\delta^2r^2, \ \
|ab(x+c)|_p\leq\alpha\beta\delta r, \ \ |bcx|_p=\alpha\beta r^2, \
\ |b^2|_p=\alpha^2\beta^2.$$

If $r\in A_1$, then $\max\{\delta^2r^2, \,
\alpha\beta\delta r, \, \alpha\beta r^2, \,
\alpha^2\beta^2\}=\alpha^2\beta^2$ and $$|(x-\hat x_1)(x-\hat
x_2)(c-\hat x_1)(c-\hat x_2)|_p=\alpha^2\beta^2.$$
Using this equality by (\ref{ab1}) we get $|f(x)-f(c)|_p=|x-c|_p\leq\r.$

If $\alpha<\beta=\delta$, then $r\in A_2$. Consequently,
$r\in[0,\alpha)$ or $r\in(\alpha,\beta)$.

Let $0\leq r<\alpha$. Then
$|f(x)-f(c)|_p=|x-c|_p\leq\r.$

Let $\alpha<r<\beta$. Then $\max\{\delta^2r^2, \,
\alpha\beta\delta r, \, \alpha\beta r^2, \,
\alpha^2\beta^2\}=\delta^2r^2$ and $$|(x-\hat x_1)(x-\hat
x_2)(c-\hat x_1)(c-\hat x_2)|_p=r^2\beta^2.$$

Consequently $|f(x)-f(c)|_p=|x-c|_p\leq\r.$ This completes the
proof.
\end{proof}

Recall that $S_r(0)$ is invariant with respect to $f$ iff $r\in
A$.
\begin{lemma}{\label{ab2}}
If $c\in S_r(0)$, where $r\in A$, then
$$|f(c)-c|_p=\left\{\begin{array}{ll}
{r^3\over{\alpha\beta}}, \ \ \mbox{if} \ \ r<\alpha\\[2mm]
{r^2\over{\beta}}, \ \ \mbox{if} \ \ \alpha<r<\beta
\end{array}\right.$$
\end{lemma}
\begin{proof}
Follows from the following equality
$$|f(c)-c|_p=\left|{{-c^3}\over{(c-\hat x_1)(c-\hat x_2)}}\right|_p=\left\{\begin{array}{ll}
{r^3\over{\alpha\beta}}, \ \ \mbox{if} \ \ r<\alpha\\[2mm]
{r^2\over{\beta}}, \ \ \mbox{if} \ \ \alpha<r<\beta.
\end{array}\right.$$
\end{proof}

By Lemma \ref{ab2} we have that $|f(c)-c|_p$ depends on $r$, but
does not depend on $c\in S_r(0)$ itself, therefore we define $\r(r)=|f(c)-c|_p$, if
$c\in S_r(0)$.

\begin{thm}{\label{ca}} If $c\in S_r(0), \, r\in A$ then
\begin{itemize}
\item[1.] For any $n\geq 1$ the following equality holds
\begin{equation}\label{en}
|f^{n+1}(c)-f^n(c)|_p=\r(r).
\end{equation}
\item[2.] $f(V_{\r(r)}(c))=V_{\r(r)}(c).$
\item[3.] If for some $\theta>0$ the ball $V_{\theta}(c)\subset S_r(0)$ is an invariant for $f$,
then $$\theta\geq \r(r).$$
\end{itemize}
\end{thm}
\begin{proof}
1. Since $S_r(0)$ is an invariant of $f$, for any $c\in
S_r(0)$ and  $n\geq 1$ we have $f^n(c)\in S_r(0)$. Take
$x=f(c)$ then by (\ref{ab1}) and proof of Lemma \ref{ab} we get
$|f(x)-f(c)|_p=|x-c|_p$, consequently
$|f^{2}(c)-f(c)|_p=|f(c)-c|_p=\r(r)$. Thus for $n=1$ the equality (\ref{en}) is true.
We use the method of mathematical induction. Assume the equality (\ref{en}) is true for $n=k$, i.e.,
$$|f^{k+1}(c)-f^k(c)|_p=|f(c)-c|_p=\r(r).$$
We shall prove it for $n=k+1$.
Denote $x_1=f^{k+1}(c)$ and $c_1=f^k(c)$ then by (\ref{ab1}) and
proof of Lemma \ref{ab} we get
$|f(x_1)-f(c_1)|_p=|x_1-c_1|_p$. Therefore
$$|f(x_1)-f(c_1)|_p=|f^{k+2}(c)-f^{k+1}(c)|_p=|f^{k+1}(c)-f^k(c)|_p=|f(c)-c|_p=\r(r).$$
This completes the proof.

2. Since $|f(c)-c|_p=\r(r)$, we have $f(c)\in V_{\r(r)}(c)$ and $c\in V_{\r(r)}(f(c))$. By Lemma \ref{ab} and the fact that any point of a ball is its center we get
$$f(V_{\r(r)}(c))=V_{\r(r)}(f(c))=V_{\r(r)}(c).$$

3. Let $V_{\theta}(c)\subset S_r(0)$ and  the ball $V_{\theta}(c)$
is an invariant for $f$, then $f(c)\in V_{\theta}(c)$, i.e.
$|f(c)-c|_p\leq\theta$. We have $\r(r)=|f(c)-c|_p$ for every $c\in
S_r(0)$. So then $\r(r)\leq\theta$.
\end{proof}

For each $r\in A$ consider a measurable space $(S_r(0),\mathcal
B)$, here $\mathcal B$ is the algebra generated by closed
subsets of $S_r(0)$. Every element of $\mathcal B$ is a union of
some balls $V_{\r}(c)$.

A measure $\bar\mu:\mathcal B\rightarrow
\mathbb{R}$ is said to be \emph{Haar measure} if it is defined by
$\bar\mu(V_{\r}(c))=\r$. We consider normalized Haar measure:
$$\mu(V_{\r}(c))={{\bar\mu(V_{\r}(c))}\over{\bar\mu(S_r(0))}}={{\r}\over{r}}, \ r>0.$$

By Lemma \ref{ab} we conclude that $f$ preserves the measure
$\mu$, i.e.
\begin{equation}{\label{ab3}}
\mu(f(V_{\r}(c)))=\mu(V_{\r}(c)).
\end{equation}

Consider the dynamical system $(X,T,\mu)$, where $T:X\rightarrow X$ is
a measure preserving transformation, and $\mu$ is a measure.
We say that the dynamical system is {\it ergodic} (or alternatively that $T$ is ergodic with respect to $\mu$ or that $\mu$
is ergodic with respect to $T$) if for every invariant set $V$ we have $\mu(V)=0$ or $\mu(V)=1$ (see \cite{Wal}).

\begin{thm}
The dynamical system $(S_r(0), f, \mu)$ is not ergodic for any
$r\in A$, $r>0$ (i.e. for any invariant sphere $S_r(0)$). Here $\mu$ is the normalized Haar measure.
\end{thm}
\begin{proof}
If a sphere $S_r(0)$ is invariant for $f$, then $0\leq r<\alpha$
or $\alpha<r<\beta$. By the part 2 of Theorem \ref{ca}, the ball
$V_{\r(r)}(c)$ is invariant for any $c\in S_r(0)$. Using Lemma
\ref{ab2} we get 
$$\mu(V_{\r(r)}(c))={{\r(r)}\over
r}=\left\{\begin{array}{ll}
{r^2\over{\alpha\beta}}, \ \ \mbox{if} \ \ 0\leq r<\alpha\\[2mm]
{r\over{\beta}}, \ \ \mbox{if} \ \ \alpha<r<\beta
\end{array}\right.$$
Since $0<r<\alpha\leq\beta$ we have  $0<{r^2\over{\alpha\beta}}<1$. If $\alpha<r<\beta$ then $0<{r\over\beta}<1$.
Therefore the dynamical system $(S_r(0), f, \mu)$ is not ergodic for all $r\in A$.

\end{proof}


\begin{thebibliography}{999}

\bibitem{ARS} S. Albeverio, U.A. Rozikov, I.A. Sattarov. $p$-adic $(2,1)$-rational dynamical systems. {\it Jour. Math. Anal. Appl.} {\bf 398}(2) (2013), 553--566.

\bibitem{AKK} S. Albeverio, P. E. Kloeden, A. Khrennikov, Human memory as a $p$-adic dynamical system, {\it Theor. Math. Phys}. {\bf 114}(3) (1998), 1414--1422.

\bibitem{A2} S. Albeverio, A. Khrennikov, B. Tirozzi, S. De Smedt, $p$-adic dynamical systems, {\it Theor. Math. Phys}. {\bf 114}(3) (1998), 276--287.

\bibitem{AK09} V. Anashin, A. Khrennikov. {\it Applied Algebraic Dynamics}, V. 49,
de Gruyter Expositions in Mathematics. Walter de Gruyter, Berlin, New York, 2009.

\bibitem{AKY11} V. S. Anashin, A. Yu. Khrennikov, and E. I. Yurova. Characterization of ergodicity of
$p$-adic dynamical systems by using van der Put basis. {\it Doklady Mathematics}, {\bf 83}(3) (2011), 306--
308.

\bibitem{Ana10} V. Anashin.  Non-Archimedean ergodic theory and pseudorandom generators. {\it The Computer
Journal}, {\bf 53}(4) (2010), 370--392.

\bibitem{CS} G. Call and J. Silverman, Canonical height on varieties with morphisms,
{\it Compositio Math}. {\bf 89}(1993), 163-205.

\bibitem{E} A. Escassut, Analytic Elements in p-Adic Analysis, World Scientific, River Edge, N. J. (1995).

\bibitem{FL11} A.-H. Fan and L.-M. Liao.  On minimal decomposition of $p$-adic polynomial dynamical
systems. {\it Adv. Math.}, {\bf 228} (2011), 2116--2144.

\bibitem{FF} A. Fan, S. Fan, L. Liao, Y. Wang, On minimal decomposition of $p$-adic homographic dynamical systems. {\it Adv. Math}. {\bf 257} (2014), 92--135.

\bibitem{KN} A. Yu. Khrennikov, M. Nilsson,  p-Adic Deterministic and Random Dynamics, Math. Appl., Vol. 574:
(Kluwer Acad. Publ., Dordreht, 2004).

\bibitem{K0} A.Yu. Khrennikov, p-Adic description of chaos, in: E. Alfinito, M. Boti (Eds.), Nonlinear Physics: Theory and
Experiment, WSP, Singapore, 1996, pp. 177–184.

\bibitem{KMa} M. Khamraev, F.M. Mukhamedov, On a class of rational $p$-adic dynamical systems, {\it J. Math. Anal. Appl}. {\bf 315} (1) (2006), 76-89.

\bibitem{Ko} N.Koblitz, {\it  $p$-adic numbers, $p$-adic analysis and zeta-function}
Springer, Berlin, 1977.

\bibitem{F} F.M. Mukhamedov, On the chaotic behavior
of cubic $p$-adic dynamical systems, {\it Math. Notes}. {\bf 83}(3) (2008), 428-–431.

\bibitem{M1} F.M. Mukhamedov, U.A. Rozikov,  On rational $p$-adic dynamical systems.
{\it Methods of Func. Anal. and Topology.} {\bf 10}(2) (2004), 21--31.

\bibitem{UF} F.M. Mukhamedov, U.A. Rozikov, A plynomial $p$-adic dynamical system. {\it Theor. Math. Phys}. {\bf 170}(3) (2012), 376-–383.

\bibitem{MM} F.M. Mukhamedov, J.F.F. Mendes, On the chaotic behavior of a generalized logistic
$p$-adic dynamical system. {\it J. Differential Equations}, {\bf 243} (2007), 125-–145.

\bibitem{PJS} H.-O.Peitgen, H.Jungers and D.Saupe, {\it Chaos Fractals}, Springer,
Heidelberg-New York, 1992.

\bibitem{Rob} A. M. Robert, A Course of p-Adic Analysis (Grad. Texts Math., Vol. 198), Springer, New York (2000).

\bibitem{RS} U.A. Rozikov, I.A. Sattarov. On a non-linear $p$-adic dynamical system. {\it $p$-Adic Numbers, Ultrametric Analysis and Applications}, {\bf 6}(1) (2014), 53--64.

\bibitem{S} I.A. Sattarov. $p$-adic $(3,2)$-rational dynamical
systems. {\it $p$-Adic Numbers, Ultrametric Analysis and
Applications}, {\bf 7}(1) (2015), 39--55.

\bibitem{Wal} P.Walters, {\it An introduction to ergodic theory.}
 Springer, Berlin-Heidelberg-New York, (1982).
\end{thebibliography}
\end{document}